\documentclass[12pt]{amsart}
\usepackage{amssymb}
\usepackage[mathscr]{eucal}
\usepackage{amscd}
\usepackage{amsfonts}
\usepackage{amsmath, amsthm, amssymb}
\usepackage{latexsym}
\usepackage[dvips]{graphics}
 \theoremstyle{plain}
 \newtheorem{thm}{Theorem}
 \newtheorem{cor}{Corollary}

\theoremstyle{definition}
 
 \newtheorem{exmp}{Example}
\theoremstyle{remark}
 \newtheorem{rem}{Remark}

\allowdisplaybreaks

\begin{document}
\title[A generalized Suzuki
Berinde]
{ A generalized Suzuki
Berinde contraction that characterizes Banach spaces}
\author[Mujahid Abbas et al.]{Mujahid Abbas$^1$, Rizwan Anjum $^{2}$ and Vladimir Rako\v cevi\' c$^{3,\ast}$}
 \thanks{$^{\ast}$Corresponding author: V. Rakocevic}
 \maketitle
\begin{center}
{\footnotesize $^1$Department of Mathematics, Government College University Katchery
Road, Lahore 54000, Pakistan and Department of Mathematics and Applied
Mathematics, University of Pretoria Hatfield 002, Pretoria, South Africa  \\
e-mail: abbas.mujahid@gmail.com\\
$^2$Abdus Salam School of Mathematical Sciences, GC University Lahore,
Pakistan. \\[0pt]
e-mail: rizwananjum1723@gmail.com \\
$^3$Department of Mathematics, University of Ni\v s, Faculty of Sciences and Mathematics, Vi\v segradska 33, Ni\v s 18000, Serbia.\\
 e-mail:  vrakoc@sbb.rs}
\end{center}

{{\noindent {\bf Abstract.}} We introduce a large class of contractive mappings, called Suzuki Berinde
type contraction. We show that any Suzuki Berinde type contraction has a
fixed point and characterizes the completeness of the underlying normed
space. A fixed point theorem for multivalued mapping is also obtained. These
results unify, generalize and complement various known comparable results in
the literature.


 }

\commby{}

\section{Introduction and Preliminaries}

Let $(X,d)$ be a metric space. A mapping $T:X\rightarrow X$ is called
contraction mapping if there exists $r\in \lbrack 0,1)$ such that for all $%
x,y\in X,$ we have%
\begin{equation*}
d(Tx,Ty)\leq rd(x,y).
\end{equation*}%
A mapping $T:X\rightarrow X$ is called contractive mapping if for all $%
x,y\in X$ with $x\neq y,$ we have%
\begin{equation*}
d(Tx,Ty)<d(x,y).
\end{equation*}%
An element $x\in X$ is called a fixed point of $T$ if $x=Tx.$

A selfmapping $T$ is called Picard operator if there exists a unique point $%
p $ in $X$ such that $Tp=p$ and $T^{n}x\rightarrow p$ for all $x$ in $X.$

A well known Banach Contraction Principle or theorem of Picard Banach
Cassioppoli reads as follows:

\begin{thm}
\cite{Bamach}\label{.1} Let $(X,d)$ be complete metric space and $T$ a
contraction mapping on $X.$ Then $T$ is a Picard operator.
\end{thm}

Define a nonincreasing function $f$ from $[0,1)$ onto $(\frac{1}{2},1]$ by
\begin{equation*}
f(r)=%
\begin{cases}
1 & \text{if $0\leq r\leq \frac{\sqrt{5}-1}{2}$}, \\
\frac{1-r}{r^{2}} & \text{if $\frac{\sqrt{5}-1}{2}<r<\frac{1}{\sqrt{2}}$},
\\
\frac{1}{1+r} & \text{if $\frac{1}{\sqrt{2}}\leq r<1.$}%
\end{cases}%
\end{equation*}

Suzzuki \cite{SUZUKI 2} proved the following Theorem, which is an
interesting generalization of the Theorem \ref{.1}.

\begin{thm}
\cite{SUZUKI 2}\label{90} Let $(X,d)$ be a complete metric space and $T$ a
selfmapping on $X.$Assume that there exists $r\in \lbrack 0,1)$ such that
for all $x,y\in X,$%
\begin{equation*}
f(r)d(x,Tx)\leq d(x,y)\text{ implies that }d(Tx,Ty)\leq rd(x,y).
\end{equation*}

Then $T$ is a Picard operator$.$
\end{thm}

The following is a well known Edelstein fixed point theorem .

\begin{thm}
\cite{Edel}\label{ss} Let $(X,d)$ be compact metric space and $T$ a
contractive mapping on $X.$ Then $T$ has a unique fixed point.
\end{thm}

Suzuki \cite{Suzuki} proved a variant of Edelstein Theorem \ref{ss} as
follows.

\begin{thm}
\cite{Suzuki}\label{101} Let $(X,d)$ be a compact metric space and $%
T:X\rightarrow X.$ Assume that for all $x,y\in X,$
\begin{equation*}
\frac{1}{2}d(x,Tx)<d(x,y)\text{ implies that }d(Tx,Ty)<d(x,y).
\end{equation*}

Then $T$ has a unique fixed point.
\end{thm}

The aim of this paper is to prove the generalization of Theorem \ref{90} in
the setting of a Banach space and characterizes the completeness of
underlying space. Moreover, we shall also prove generalization of Theorem %
\ref{101} in the setting of a compact normed space. A multivalued fixed
point theorem is also proved which generalizes Theorem $2$ in \cite{KIKKAWA}%
, and Theorem $2.3$ in \cite{Beg} in the setting of Banach space.

\section{Main Result}

We start with the following theorem which is generalization of the Theorem $2
$ (\cite{SUZUKI 2}) in the setting of a Banach space.

\begin{thm}
\label{bot} Let $(X,||.||)$ be Banach space and $T$ a selfmapping on $X.$ If
there exists $b\in \lbrack 0,\infty )$ and $\theta \in \lbrack 0,b+1)$ with $%
\lambda =\frac{1}{b+1}$ such that for any $x,y\in X$
\begin{equation}
\psi (r)||x-Tx||\leq ||x-y||  \label{pp1}
\end{equation}%
implies that
\begin{equation}
||b(x-y)+Tx-Ty||\leq \theta ||x-y||,  \label{pp2}
\end{equation}%
where, $\ \theta \lambda =r$ and $\psi $ is a nonincreasing function from $%
[0,1)$ onto $[0,1)$ given by
\begin{equation*}
\psi (r)=%
\begin{cases}
\lambda & \text{if $0\leq r\leq \frac{\sqrt{5}-1}{2}$}, \\
\frac{\lambda (1-r)}{r^{2}} & \text{if $\frac{\sqrt{5}-1}{2}<r<\frac{1}{%
\sqrt{2}}$}, \\
\frac{\lambda }{1+r} & \text{if $\frac{1}{\sqrt{2}}\leq r<1.$}%
\end{cases}%
\end{equation*}%
Then $T$ has unique fixed point.
\end{thm}

\begin{proof}
We divide the proof into the following two cases.\newline
\textsc{Case 1.} Suppose that $b>0.$ Clearly, $0<\lambda <1.$ We can write $%
\psi (r)=\lambda f(r),$ where $f$ is a nonincreasing function from $[0,1)$
onto $(\frac{1}{2},1]$ given by
\begin{equation*}
f(r)=%
\begin{cases}
1 & \text{if $0\leq r\leq \frac{\sqrt{5}-1}{2}$}, \\
\frac{1-r}{r^{2}} & \text{if $\frac{\sqrt{5}-1}{2}<r<\frac{1}{\sqrt{2}}$},
\\
\frac{1}{1+r} & \text{if $\frac{1}{\sqrt{2}}\leq r<1.$}%
\end{cases}%
\end{equation*}%
Then (\ref{pp1}) and (\ref{pp2}) become
\begin{eqnarray*}
\lambda f(r)||x-Tx|| &\leq &||x-y||\text{ and} \\
||T_{\lambda }x-T_{\lambda }y|| &\leq &r||x-y||,
\end{eqnarray*}%
respectively. In this case, an implicative contractive condition in the
statement of the Theorem reduces to the following form:
\begin{equation}
f(r)||x-T_{\lambda }x||\leq ||x-y||  \label{z3}
\end{equation}%
implies that
\begin{equation}
||T_{\lambda }x-T_{\lambda }y||\leq r||x-y||,  \label{z4}
\end{equation}%
for all $x,y\in X.$ Since $f(r)\leq 1,$ $f(r)||x-T_{\lambda }x||\leq
||x-T_{\lambda }x||$ holds for every $x\in X.$ By (\ref{z3}), we have
\begin{equation}
||T_{\lambda }x-T_{\lambda }^{2}x||\leq r||x-T_{\lambda }x||,  \label{z5}
\end{equation}%
for $x\in X.$ Now we fix $u\in X$ and define a sequence $\{u_{n}\}$ in $X$
by $u_{n}=T_{\lambda }^{n}u$. Then (\ref{z5}) gives that
\begin{equation*}
||u_{n}-u_{n+1}||\leq r^{n}||x-T_{\lambda }x||.
\end{equation*}%
Thus, $\sum_{n=1}^{\infty }||u_{n}-u_{n+1}||<\infty ,$ and hence $\{u_{n}\}$
is Cauchy sequence in a Banach space $X$. Assume that there exists a point $%
u\in X$ such that $\{u_{n}\}$ converges to $u$ as $n\rightarrow \infty .$
For $x\in X\setminus \{z\},$ there exists $n_{0}\in \mathbb{N}$ such that $%
||u_{n}-z||\leq \frac{||x-z||}{3}$ for all $n\in \mathbb{N}$ with $n\geq
n_{0}.$ Then we have
\begin{align*}
f(r)||u_{n}-T_{\lambda }u_{n}||& \leq ||u_{n}-T_{\lambda
}u_{n}||=||u_{n}-u_{n+1}|| \\
& \leq ||u_{n}-z||+||u_{n+1}-z|| \\
& \leq \frac{2}{3}||x-z||=||x-z||-\frac{||x-z||}{3} \\
& \leq ||x-z||-||u_{n}-z||\leq ||u_{n}-x||.
\end{align*}%
By using (\ref{z3}), we obtain that
\begin{equation*}
||u_{n+1}-T_{\lambda }x||\leq r||u_{n}-x||,\ \ \ \forall \ n\geq n_{0}.
\end{equation*}%
On taking limit as $n\rightarrow \infty ,$ we have
\begin{equation}
||T_{\lambda }x-z||\leq r||x-z||,\ \ \ \forall \ x\in X\setminus \{z\}.
\label{z6}
\end{equation}%
We now show that $T_{\lambda }^{k}z=z,$ for some $k_{0}\in \mathbb{N}$ . On
the contrary suppose that $T_{\lambda }^{k}z\neq z,$ for all $k\in \mathbb{N}%
.$ From (\ref{z6}), we have
\begin{equation}
||T_{\lambda }^{k+1}z-z||\leq r^{k}||T_{\lambda }z-z||,\ \ \ \forall \ k\in
\mathbb{N}.  \label{z7}
\end{equation}%
We consider the following three cases.

\begin{description}
\item[i] $0\leq r\leq\frac{\sqrt{5}-1}{2}$,

\item[ii] $\frac{\sqrt{5}-1}{2}< r<\frac{1}{\sqrt{2}}$

\item[iii] $\frac{1}{\sqrt{2}}\leq r<1.$
\end{description}

{Case i.} When $0\leq r\leq \frac{\sqrt{5}-1}{2},$ we have $r^{2}+r-1\leq 0$.%
$.$ If we assume that $||T_{\lambda }^{2}z-z||<||T_{\lambda
}^{2}z-T_{\lambda }^{3}z||,$ then we obtain
\begin{align*}
||T_{\lambda }z-z||& \leq ||z-T_{\lambda }^{2}z||+||T_{\lambda }z-T_{\lambda
}^{2}z|| \\
& <||T_{\lambda }^{2}z-T_{\lambda }^{3}z||+||T_{\lambda }z-T_{\lambda
}^{2}z|| \\
& \leq r^{2}||z-T_{\lambda }z||+r||z-T_{\lambda }z|| \\
& \leq ||z-T_{\lambda }z||,
\end{align*}%
a contradiction. So we have
\begin{equation*}
f(r)||T_{\lambda }^{2}z-T_{\lambda }^{3}z||\leq ||T_{\lambda }^{2}z-z||.
\end{equation*}%
Now by (\ref{z3}) and (\ref{z7}), we have
\begin{align*}
||z-T_{\lambda }z||& \leq ||z-T_{\lambda }^{3}z||+||T_{\lambda
}^{3}z-T_{\lambda }z|| \\
& \leq r^{2}||z-T_{\lambda }z||+r||T_{\lambda }^{2}z-z|| \\
& \leq r^{2}||z-T_{\lambda }z||+r^{2}||T_{\lambda }z-z|| \\
& =2r^{2}||z-T_{\lambda }z|| \\
& <||z-T_{\lambda }z||,
\end{align*}%
a contradiction.\newline
{Case ii.} Suppose that $2r^{2}<1.$ If
\begin{equation*}
||T_{\lambda }^{2}z-z||<f(r)||T_{\lambda }^{2}z-T_{\lambda }^{3}z||,
\end{equation*}%
then we have
\begin{align*}
||z-T_{\lambda }z||& \leq ||z-T_{\lambda }^{2}z||+||T_{\lambda }z-T_{\lambda
}^{2}z|| \\
& <f(r)||T_{\lambda }^{2}z-T_{\lambda }^{3}z||+||T_{\lambda }z-T_{\lambda
}^{2}z|| \\
& \leq f(r)r^{2}||z-T_{\lambda }z||+r||z-T_{\lambda }z||=||z-T_{\lambda }z||,
\end{align*}%
a contradiction. Hence
\begin{equation*}
f(r)||T_{\lambda }^{2}z-T_{\lambda }^{3}z||\leq ||T_{\lambda }^{2}z-z||.
\end{equation*}%
As in the case $i,$ we can prove that%
\begin{equation*}
||z-T_{\lambda }z||\leq 2r^{2}||z-T_{\lambda }z||<||z-T_{\lambda }z||
\end{equation*}

which gives a contradiction.\newline
{Case iii.} Suppose that $\frac{1}{\sqrt{2}}\leq r<1.$ Note that for $x,y\in
X,$ either%
\begin{equation*}
f(r)||x-T_{\lambda }x||\leq ||x-y||
\end{equation*}

or \
\begin{equation*}
f(r)||T_{\lambda }x-T_{\lambda }^{2}x||\leq ||T_{\lambda }x-y||
\end{equation*}%
hold. Indeed if%
\begin{eqnarray*}
f(r)||x-T_{\lambda }x|| &>&||x-y||\text{ and } \\
f(r)||T_{\lambda }x-T_{\lambda }^{2}x|| &>&||T_{\lambda }x-y||,
\end{eqnarray*}%
then we have
\begin{align*}
||x-T_{\lambda }x||& \leq ||x-y||+||T_{\lambda }x-y|| \\
& <f(r)(||x-T_{\lambda }x||+||T_{\lambda }x-T_{\lambda }^{2}x||) \\
& \leq f(r)(||x-T_{\lambda }x||+r||x-T_{\lambda }x||) \\
& =||x-T_{\lambda }x||,
\end{align*}%
a contradiction. Since either%
\begin{eqnarray*}
f(r)||u_{2n}-u_{2n+1}|| &\leq &||u_{2n}-z||\text{ or} \\
\ f(r)||u_{2n+2}-u_{2n+1}|| &\leq &||u_{2n+1}-z||
\end{eqnarray*}%
hold for every $n\in \mathbb{N}.$ Thus, either%
\begin{eqnarray*}
||u_{2n+1}-T_{\lambda }z|| &\leq &r||u_{2n}-z||\text{ or} \\
\ ||u_{2n+2}-T_{\lambda }z|| &\leq &r||u_{2n+1}-z||
\end{eqnarray*}%
hold for every $n\in \mathbb{N}$. As $\{u_{n}\}$ converges to $z$, the above
inequalities imply that there exists a subsequence of $\{u_{n}\}$ which
converges to $T_{\lambda }z$ which further implies that $T_{\lambda }z=z$, a
contradiction. Therefore in all the three cases, there exists $k\in \mathbb{N%
}$ such that $T_{\lambda }^{k}z=z$. Since $\{T_{\lambda }^{k}z\}$ is a
Cauchy sequence, we obtain $T_{\lambda }z=z.$ That is, $z$ is a fixed point
of $T_{\lambda },$ and hence $Tz=z.$ The uniqueness of a fixed point follows
from (\ref{z6}). \newline
\newline
\textsc{Case 2.} $b=0.$ In this case, we have $\lambda =1$ and $\theta =r.$
Then, for all $x,y\in X$%
\begin{equation*}
f(r)||x-Tx||\leq ||x-y||
\end{equation*}%
implies that%
\begin{equation*}
\ ||Tx-Ty||\leq r||x-y||.
\end{equation*}%
It follows from Theorem $2$ of \cite{SUZUKI 2} that $T$ has unique fixed
point.
\end{proof}

\begin{exmp}
Let $X=\mathbb{R}^{2}$ be equipped with the norm $||.||$ given by $%
||(x_{1},x_{2})||=\sqrt{x_{1}^{2}+x_{2}^{2}}.$ Define a mapping $T$ on $X$
by
\begin{equation*}
T(x_{1},x_{2})=%
\begin{cases}
(0,0) & \text{if $(x_{1},x_{2})=\mathbb{R}^{2}\setminus \{(4,5),(5,4)\}$} \\
(4,0) & \text{if $(x_{1},x_{2})=(4,5)$} \\
(0,4) & \text{if $(x_{1},x_{2})=(5,4)$}%
\end{cases}%
\end{equation*}%
Note that, for $b=1,$ $\theta =1$ and $r=\frac{1}{2},$ $T$ satisfies the
assumption in Theorem \ref{bot}. Indeed,%
\begin{equation*}
||b(x-y)+Tx-Ty||\leq ||x-y||\text{ if }x,y\in \mathbb{R}^{2}\setminus
\{(4,5),(5,4)\}.
\end{equation*}
Since
\begin{equation*}
\psi (\frac{1}{2})||(4,5)-T(4,5)||=\frac{1}{2}||(4,5)-(4,0)||=\frac{5}{2}%
>2=||(4,5)-(5,4)||.
\end{equation*}
\end{exmp}

As a Corollary of our result; if we put $b=0$ in Theorem \ref{bot}, we
obtain Theorem $2$ of \cite{SUZUKI 2} in the setting of Banach spaces.

\begin{cor}
\label{001} Let $(X,||.||)$ be Banach space and $T$ a mapping on $X.$ Define
a nonincreasing function $f$ from $[0,1)$ onto $(\frac{1}{2},1]$ by
\begin{equation*}
f(r)=%
\begin{cases}
1 & \text{if $0\leq r\leq \frac{\sqrt{5}-1}{2}$}, \\
\frac{1-r}{r^{2}} & \text{if $\frac{\sqrt{5}-1}{2}<r<\frac{1}{\sqrt{2}}$},
\\
\frac{1}{1+r} & \text{if $\frac{1}{\sqrt{2}}\leq r<1.$}%
\end{cases}%
\end{equation*}%
Assume that if there exists $r\in \lbrack 0,1)$ such that for any $x,y\in X$
\begin{equation}
f(r)||x-Tx||\leq ||x-y||  \label{z1}
\end{equation}%
implies
\begin{equation}
||Tx-Ty||\leq r||x-y||.  \label{z2}
\end{equation}%
Then there exists a unique fixed point of $T.$
\end{cor}

We prove the following theorem, which is the generalization of the Theorem %
\ref{101} in the setting of compact normed spaces.

\begin{thm}
\label{main} Let $(X,||.||)$ be compact normed space and $T:X\rightarrow X$.
If there exists $b\in \lbrack 0,\infty )$ with $\lambda =\frac{1}{b+1}$ such
that for any $x,y\in X$
\begin{equation}
\frac{\lambda }{2}||x-Tx||<||x-y||  \label{rizwan}
\end{equation}%
implies that
\begin{equation}
||b(x-y)+Tx-Ty||<||x-y||.  \label{anjum}
\end{equation}%
Then $T$ has a fixed point.
\end{thm}

\begin{proof}
We divide the proof into the following two cases.\newline
\textsc{Case 1.} Suppose that $b>0.$ Take $\lambda =\frac{1}{b+1}.$ Clearly,
$0<\lambda <1.$ In this case, the given assumption becomes; if for $x,y\in X
$
\begin{equation}
\frac{1}{2}||x-T_{\lambda }x||<||x-y||  \label{1}
\end{equation}%
implies
\begin{equation}
||T_{\lambda }x-T_{\lambda }y||<\lambda ||x-y||<||x-y||.  \label{2}
\end{equation}%
We put

\begin{equation*}
\beta =\inf \{\frac{1}{2}||x-T_{\lambda }x||:\ x\in X\}
\end{equation*}

and choose a sequence $\{x_{n}\}\in X$ such that $\lim_{n\rightarrow \infty
}||x_{n}-T_{\lambda }x_{n}||=\beta .$ Since $X$ is compact, without any loss
of generality, we may assume that $\{x_{n}\}$ and $\{T_{\lambda }{x_{n}}\}$
converge to some element $v,w\in X,$ respectively. We now show that $\beta
=0.$ On then contrary suppose that $\beta >0.$ We have
\begin{equation*}
\lim_{n\rightarrow \infty }||x_{n}-T_{\lambda
}x_{n}||=||v-w||=\lim_{n\rightarrow \infty }||x_{n}-w||=\beta
\end{equation*}%
We may now choose $n_{0}\in \mathbb{N}$ such that%
\begin{equation*}
\frac{2\beta }{3}<||x_{n}-w||\text{ and }||x_{n}-T_{\lambda }x_{n}||<\frac{%
4\beta }{3}.
\end{equation*}%
Thus, $\frac{1}{2}||x_{n}-T_{\lambda }x_{n}||<||x_{n}-w||$ for $n\geq n_{0}.$
From (\ref{1}) and (\ref{2}) , we have
\begin{equation*}
||T_{\lambda }x_{n}-T_{\lambda }w||<||x_{n}-w||,\ \ \ \ \forall \ n\geq
n_{0}.
\end{equation*}%
This implies that
\begin{align*}
||w-T_{\lambda }w||& =\lim_{n\rightarrow \infty }||T_{\lambda
}x_{n}-T_{\lambda }w|| \\
& <\lim_{n\rightarrow \infty }||x_{n}-w||=\beta ,
\end{align*}%
that is,
\begin{equation*}
||w-T_{\lambda }w||<\beta .
\end{equation*}%
From the definition of $\beta ,$ we obtain $||w-T_{\lambda }w||=\beta .$
Since $\frac{1}{2}||w-T_{\lambda }w||<||w-T_{\lambda }w||,$ we have%
\begin{equation*}
||T_{\lambda }w-T_{\lambda }^{2}w||<||w-T_{\lambda }w||=\beta ,
\end{equation*}

which contradicts the definition of $\beta $ and hence $\beta =0.$ Next, we
show that $T$ has a fixed point. Assume on contrary that $T$ does not have
fixed point. Note that
\begin{equation*}
||T_{\lambda }x_{n}-T_{\lambda }^{2}x_{n}||<||x_{n}-T_{\lambda }x_{n}||,
\end{equation*}%
holds for every $n\in \mathbb{N}$ because
\begin{equation*}
\frac{1}{2}||x_{n}-T_{\lambda }x_{n}||<||x_{n}-T_{\lambda }x_{n}||.
\end{equation*}%
Thus,
\begin{equation*}
\lim_{n\rightarrow \infty }||v-T_{\lambda
}x_{n}||=||v-w||=\lim_{n\rightarrow \infty }||x_{n}-T_{\lambda
}x_{n}||=\beta =0,
\end{equation*}
implies that $\{T_{\lambda }x_{n}\}$ also converges to $v.$ Also,
\begin{align*}
\lim_{n\rightarrow \infty }||v-T_{\lambda }^{2}x_{n}||& \leq
\lim_{n\rightarrow \infty }(||v-T_{\lambda }x_{n}||+||T_{\lambda
}x_{n}-T_{\lambda }^{2}x_{n}||) \\
& <\lim_{n\rightarrow \infty }(||v-T_{\lambda }x_{n}||+||x_{n}-T_{\lambda
}x_{n}||) \\
& =||v-v||+||v-v||=0.
\end{align*}%
Thus, $\{T_{\lambda }^{2}x_{n}\}$ converges to $v.$ If%
\begin{equation*}
||x_{n}-v||\leq \frac{1}{2}||x_{n}-T_{\lambda }x_{n}||\text{ and }%
||T_{\lambda }x_{n}-v||\leq \frac{1}{2}||T_{\lambda }x_{n}-T_{\lambda
}^{2}x_{n}||,
\end{equation*}%
then we have
\begin{align*}
||x_{n}-T_{\lambda }x_{n}||& \leq ||x_{n}-v||+||T_{\lambda }x_{n}-v|| \\
& \leq \frac{1}{2}||x_{n}-T_{\lambda }x_{n}||+\frac{1}{2}||T_{\lambda
}x_{n}-T_{\lambda }^{2}x_{n}|| \\
& <\frac{1}{2}||x_{n}-T_{\lambda }x_{n}||+\frac{1}{2}||x_{n}-T_{\lambda
}x_{n}||,
\end{align*}%
that is,
\begin{equation*}
||x_{n}-T_{\lambda }x_{n}||<||x_{n}-T_{\lambda }x_{n}||
\end{equation*}%
a contradiction. Hence for every $n\in \mathbb{N},$ either%
\begin{equation*}
\frac{1}{2}||x_{n}-T_{\lambda }x_{n}||<||x_{n}-v||\text{ or }\frac{1}{2}%
||T_{\lambda }x_{n}-T_{\lambda }^{2}x_{n}||<||T_{\lambda }x_{n}-v||
\end{equation*}%
holds. Using (\ref{rizwan}), we conclude that either%
\begin{equation*}
||T_{\lambda }x_{n}-T_{\lambda }v||<||x_{n}-v||\text{ or }||T_{\lambda
}^{2}x_{n}-T_{\lambda }v||<||T_{\lambda }x_{n}-v||
\end{equation*}%
holds. Hence one of the following conditions holds:

\begin{enumerate}
\item There exists an infinite subset $I$ of $\mathbb{N}$ such that
\begin{equation}
||T_{\lambda }x_{n}-T_{\lambda }v||<||x_{n}-v||,\ \ \forall \ n\in I.\label{qw}
\end{equation}%
\item There exists an infinite subset $J$ of $\mathbb{N}$ such that
\begin{equation}
||T_{\lambda }^{2}x_{n}-T_{\lambda }v||<||T_{\lambda }x_{n}-v||,\ \ \forall
\ n\in J.  \label{qwe}
\end{equation}%
In the first case; by using (\ref{qw}) we obtain
\begin{align*}
||v-T_{\lambda }v||& =\lim_{n\in I,\ n\rightarrow \infty }||T_{\lambda
}x_{n}-T_{\lambda }v|| \\
& \le \lim_{n\in I,\ n\rightarrow \infty }||x_{n}-v|| \\
& =||v-v||=0,
\end{align*}%
which implies that $T_{\lambda }v=v$. Also in the second case, using (\ref%
{qwe}) we obtain
\begin{align*}
||v-T_{\lambda }v||& =\lim_{n\in J,\ n\rightarrow \infty }||T_{\lambda
}^{2}x_{n}-T_{\lambda }v|| \\
& \le \lim_{n\in J,\ n\rightarrow \infty }||T_{\lambda }x_{n}-v||=0.
\end{align*}%
Hence, we have shown that $v$ is a fixed point of $T_{\lambda }$ in both
cases, a contradiction. Therefore $Fix(T)\neq \emptyset .$\newline
\textsc{Case 2.} $b=0.$ In this case, the Suzuki Berinde type contraction
becomes;
\begin{equation*}
\frac{1}{2}||x-Tx||<||x-y||
\end{equation*}%
implies that
\begin{equation*}
||Tx-Ty||<\left\Vert x-y\right\Vert
\end{equation*}%
Since $\theta \in \lbrack 0,1).$ It follows from Theorem $3$ of \cite{Suzuki}%
, $Fix(T)\neq \emptyset .$
\end{enumerate}
\end{proof}

As a Corollary of our result; if we put $b=0$ in Theorem (\ref{anjum}), then
we obtain Theorem $3$ of \cite{Suzuki} in the setting of a compact normed
space.

\begin{cor}
Let $(X,||.||)$ be compact linear normed space and $T$ a mapping on $X.$
Assume that%
\begin{equation*}
\frac{1}{2}||x-Tx||<||x-y||\text{ implies that }||Tx-Ty||<||x-y||
\end{equation*}

for all $x,y\in X.$ Then $T$ has fixed point.
\end{cor}

Now we shall discuss the completeness of underlying space.\newline

\begin{thm}
\label{awara} Let $(X,\left\Vert .\right\Vert )$ be a normed space. Define $%
\psi $ as in Theorem \ref{bot}. For $b\in \lbrack 0,\infty )$ and $\theta
\in \lbrack 0,b+1)$ with $\lambda =\frac{1}{b+1},$ let $\Gamma _{b,s}$ be
the family of mappings $T$ on $X$ satisfying the following:

\begin{itemize}
\item[(a)] \ For $x,y\in X,$
\begin{equation}  \label{bot1}
s\left\Vert x-Tx\right\Vert\leq \left\Vert x-y\right\Vert
\end{equation}
implies
\begin{equation}  \label{bot2}
\left\Vert b(x-y)+Tx-Ty\right\Vert\leq \theta \left\Vert x-y\right\Vert,
\end{equation}
where $r=\theta\lambda$ and $s\in \psi(r).$
\end{itemize}

Suppose that $\Theta _{b,s}$ is the family of mappings $T$ on $X$ satisfying
$(a)$ and the following:

\begin{itemize}
\item[(b)] \ $\big((1-\lambda)I+\lambda T\big)(X)$ is countably infinite.

\item[(c)] Every subset of $\big((1-\lambda )I+\lambda T\big)(X)$ is closed,
where $I$ is the identity map on $X.$
\end{itemize}

Then the following are equivalent:

\begin{itemize}
\item[(i)] $X$ is Banach space.

\item[(ii)] Every mapping $T\in \Gamma_{b,\psi(r)}$ has a fixed point for
all $r\in [0,1).$

\item[(iii)] There exists $r\in (0,1)$ and $s\in (0,\psi (r)]$ such that
every mapping $T\in \Gamma _{b,s}$ has a fixed point.
\end{itemize}
\end{thm}

\begin{proof}
We divide the proof into the following two cases.\newline
\textsc{Case 1.} Suppose that $b>0.$ Take $\lambda =\frac{1}{b+1}.$ Clearly,
$0<\lambda <1.$ It follows form Theorem \ref{bot} that $(i)$ implies $(ii).$
Clearly, $\Theta _{b,s}\subset T\in \Gamma _{b,\psi (r)}$ for $r\in \lbrack
0,1)$ which gives that $(ii)$ implies $(iii).$ We now show that $(iii)$
implies $(i).$ On the contrary suppose that $X$ is not a Banach space. That
is, there exits a Cauchy sequence $\{x_{n}\}\in X$ which does not converge.
Define a function
\begin{equation*}
g:X\rightarrow \lbrack 0,\infty )
\end{equation*}%
by
\begin{equation*}
g(x):=\lim_{n\rightarrow \infty }\left\Vert x-x_{n}\right\Vert ,\ \ \ \
\forall \ x\in X.
\end{equation*}%
Clearly, $g$ is well defined because $\{\left\Vert x-x_{n}\right\Vert \}$ is
a Cauchy sequence for every $x\in X.$ The following are obvious:

\begin{itemize}
\item $g(x)-g(y)\leq \left\Vert x-y\right\Vert\leq g(x)+g(y)$ for $x,y\in X,$

\item $g(x)>0$ for $x\in X,$

\item $\lim_{n\to\infty}g(x_{n})=0.$
\end{itemize}

Define a mapping $T_{\lambda }$ on $X$ as follows: For each $x\in X$, since $%
f(x)>0$ and $\lim_{n\rightarrow \infty }g(x_{n})=0,$ there exist $n_{0}\in
\mathbb{N}$ satisfying
\begin{equation*}
f(x_{0})\leq \frac{s\theta \lambda }{3+s\theta \lambda }.
\end{equation*}%
Define $T:X\rightarrow X$ by
\begin{equation*}
T(x)=\frac{u_{n_{0}}-(1-\lambda )x}{\lambda }
\end{equation*}%
If we put
\begin{equation*}
T_{\lambda }x=(1-\lambda )x+\lambda Tx=u_{n_{0}}.
\end{equation*}%
Then we have%
\begin{equation*}
f(T_{\lambda }x)\leq \frac{s\theta \lambda }{3+s\theta \lambda }f(x)\text{
and }T_{\lambda }x\in \{x_{n}:\ n\in \mathbb{N}\}
\end{equation*}%
for all $x\in X.$ Also, $T_{\lambda }x\neq x$ for all $x\in X$ because $%
f(T_{\lambda }(x))<f(x).$ That is, $T_{\lambda }$ does not have fixed point.
Since $T_{\lambda }(X)\subset \{x_{n}:\ n\in \mathbb{N}\},$ $(b)$ holds.
Also it is straightforward to prove $(c)$. \newline
We now prove $(a).$ Fix $x,y\in X$ with
\begin{equation}
s\left\Vert x-Tx\right\Vert \leq \left\Vert x-y\right\Vert .  \label{jhg}
\end{equation}%
Now $s\in \psi (r)$ implies that $s=\lambda t,$ for some $t\in f(r)$ where $f
$ is defined in Theorem \ref{bot}. Thus (\ref{jhg}) becomes,
\begin{equation}
t\left\Vert x-T_{\lambda }x\right\Vert \leq \left\Vert x-y\right\Vert .
\label{fds}
\end{equation}%
In the case where $g(y)>2g(x),$ we have
\begin{align*}
\left\Vert T_{\lambda }x-T_{\lambda }y\right\Vert & \leq g(T_{\lambda
}x)+g(T_{\lambda }y) \\
& \leq \frac{s\theta \lambda }{3+s\theta \lambda }(g(x)+g(y)) \\
& \leq \frac{\theta \lambda }{3}(g(x)+g(y)) \\
& \leq \frac{\theta \lambda }{3}(g(x)+g(y))+\frac{2\theta \lambda }{3}%
(g(y)-2g(x)) \\
& =\theta \lambda (g(y)-g(x))\leq \theta \lambda \left\Vert x-y\right\Vert .
\end{align*}%
This gives that
\begin{equation}
\left\Vert T_{\lambda }x-T_{\lambda }y\right\Vert \leq \theta \lambda
\left\Vert x-y\right\Vert .  \label{qwer}
\end{equation}%
Since $\lambda =\frac{1}{b+1}.$ So, (\ref{qwer}) becomes
\begin{equation*}
\left\Vert b(x-y)+Tx-Ty\right\Vert \leq \theta \left\Vert x-y\right\Vert .
\end{equation*}%
In the other case, where $g(y)\leq 2g(x),$ we have
\begin{align*}
\left\Vert x-y\right\Vert & t\left\Vert x-T_{\lambda }x\right\Vert \geq
s(g(x)-f(T_{\lambda }x)) \\
& \geq s\biggl(1-\frac{s\theta \lambda }{3+s\theta \lambda }\biggr )g(x)=\frac{3s}{%
3+s\theta \lambda }g(x)
\end{align*}%
and hence
\begin{align*}
\left\Vert T_{\lambda }x-T_{\lambda }y\right\Vert & \leq f(T{\lambda }x)+f(T{%
\lambda }y) \\
& \leq \frac{s\theta \lambda }{3+s\theta \lambda }\big(g(x)-g(y)\big) \\
& \leq \frac{3s\theta \lambda }{3+s\theta \lambda }f(x)\leq \theta \lambda
\left\Vert x-y\right\Vert .
\end{align*}%
Since $\lambda =\frac{1}{b+1},$ we have
\begin{equation*}
\left\Vert b(x-y)+Tx-Ty\right\Vert \leq \theta \left\Vert x-y\right\Vert .
\end{equation*}%
Hence we have proved $(a),$ that is, $T\in \Theta _{b,s}.$ Since $%
Fix(T)=Fix(T_{\lambda }).$ By $(iii),$ $T$ has a fixed point, a
contradiction. Thus $X$ is Banach space.\newline
\textsc{Case 2.} $b=0.$ In this case, by Theorem (Theorem 4 \cite{SUZUKI 2})
$(i)$ implies $(ii),$ $(ii)$ implies $(iii)$ and $(iii)$ implies $(i)$. This
complete the proof.
\end{proof}

As a Corollary of our result, if we put $b=0$ in Theorem \ref{awara}, then
we can obtain Theorem $4$ of \cite{SUZUKI 2}, in the setting of normed space.

\begin{cor}
Let $(X,\left\Vert .\right\Vert )$ be a linear normed space. Define mapping $%
\psi $ as in Theorem \ref{bot}. For $b=0$, $r=\theta \in \lbrack 0,1),$ and $%
\lambda =1,$ let $\Gamma _{0,s}$ be the family of mappings $T$ on $X$
satisfying the following:

\begin{itemize}
\item[(a)] \ For $x,y\in X,$
\begin{equation}
s\left\Vert x-Tx\right\Vert \leq \left\Vert x-y\right\Vert  \label{bot1x}
\end{equation}%
implies that
\begin{equation}
\left\Vert Tx-Ty\right\Vert \leq r\left\Vert x-y\right\Vert ,  \label{bot2x}
\end{equation}%
where $s\in \psi (r).$
\end{itemize}

Let $\Theta_{0,s}$ be the family of mappings $T$ on $X$ satisfying $(a)$ and
the following:

\begin{itemize}
\item[(b)] \ $T(X)$ is countably infinite.

\item[(c)] Every subset of $T(X)$ is closed,
\end{itemize}

Then the following are equivalent:

\begin{itemize}
\item[(i)] $X$ is Banach space.

\item[(ii)] Every mapping $T\in \Gamma_{0,\psi(r)}$ has a fixed point for
all $r\in [0,1).$

\item[(iii)] There exists $r\in (0,1)$ and $s\in (0,\psi (r)]$ such that
every mapping $T\in \Gamma _{0,s}$ has a fixed point.
\end{itemize}
\end{cor}

\section{Multivalued Version}

Let $(X,d)$ be any metric space and $CB(X)$ the collection of all closed and
bounded subsets of $X.$ For $A,B\in CB(X),$ define
\begin{equation*}
D(A,B)=\sup_{a\in A}\{d(a,B)\},
\end{equation*}%
where $d(a,B)=\inf \{d(a,b)\ \ b\in B\}$ is the distance of the point $a$
from the set $B.$\newline
Define the functional $H:CB(X)\times CB(X)\rightarrow
\mathbb{R}
^{+}$ by
\begin{equation*}
H(A,B)=\max \big\{D(A,B),D(B,A)\big\}.
\end{equation*}%
\newline
The mapping $H$ is called Pompeiu-Hausdorff metric induced by $d.$ An
element $x\in X$ is called fixed point of $T,$ if $x\in T(x).$ The set of
all fixed points of a multivalued $T$ is denoted by $Fix(T),$ that is,
\begin{equation*}
Fix(T)=\{x\in X:x\in T(x)\}.
\end{equation*}

\begin{rem}
\label{remm} Let $M$ be a convex subset of a linear space $X$ and $%
T:M\rightarrow CB(M).$ Then for any $\lambda \in (0,1),$ consider a mapping $%
T_{\lambda }:M\rightarrow CB(M)$ is given by
\begin{eqnarray}
T_{\lambda }(x) &=&(1-\lambda )x+\lambda Tx  \label{r} \\
&=&\{(1-\lambda )x+\lambda s:s\in Tx\}.
\end{eqnarray}%
In other words, for each $x$ in $M,$ $T_{\lambda }(x)$ is the translation of
the set $\lambda Tx$ by the vector $(1-\lambda )x.$ Clearly,
\begin{equation*}
Fix(T_{\lambda })=Fix(T).
\end{equation*}%
Indeed, if $p\in Tp$, then $p=(1-\lambda )p+\lambda p\in T_{\lambda }p.$ On
the other hand, if $p\in T_{\lambda }p,$ then for some $s\in Tp,$ we have $%
p=(1-\lambda )p+\lambda s$ which further implies that $s=p.$
\end{rem}

\begin{thm}
\label{sds} Let $(X,||.||)$ be a Banach space and $T:X\rightarrow CB(X)$ be
multivalued map. If there exists $b\in [0,\infty)$ and $\theta\in [0,b+1)$
with $\lambda=\frac{1}{b+1}$ such that for any $x,y\in X$
\begin{equation*}
\psi(r)d(x,Tx)\leq \left\Vert x-y\right\Vert\ \ \ implies \ \ \
H(bx+Tx,by+Ty)\leq \theta \left\Vert x-y\right\Vert,
\end{equation*}%
where $\theta\lambda=r$ and $\psi$ be a strictly decreasing function from $%
[0,1)$ to $[0,1)$ by $\psi(r)=\frac{\lambda}{1+r}.$ Then $Fix(T)$ is
nonempty.
\end{thm}

\begin{proof}
We divide the proof into the following two cases.\newline
\textsc{Case 1.} Suppose that $b>0.$ Clearly, $0<\lambda <1.$ We can write $%
\psi (r)=\lambda \eta (r),$ where $\eta $ is a strictly decreasing function
from $[0,1)$ onto $(\frac{1}{2},1]$ given by $\eta (r)=\frac{1}{1+r}.$ In
this case, an implicative contractive condition in the statement of the
Theorem reduces to the following form;
\begin{equation*}
\eta (r)d(x,T_{\lambda }x)\leq \left\Vert x-y\right\Vert \ \ \text{implies
that}\ \ H(T_{\lambda }x,T_{\lambda }y)\leq r\left\Vert x-y\right\Vert .
\end{equation*}%
Clearly $T_{\lambda }$ satisfies all the conditions of Theorem $2$ of \cite%
{KIKKAWA}. Hence $Fix(T)\neq \emptyset .$ \newline
\textsc{Case 2.} Suppose that $b=0.$ In this case, we have $\lambda =1$ and $%
\theta =r.$ Then $Fix(T)\neq \emptyset $ by Theorem $2$ of \cite{KIKKAWA}.
\end{proof}

\begin{thm}
\label{ghgh} Let $(X,||.||)$ be a Banach space and $T:X\rightarrow CB(X)$.
If there exists $b\in \lbrack 0,\infty ),$ $\gamma \in (0,1)$ and $\theta
\in \lbrack 0,b+1)$ with $\lambda =\frac{1}{b+1}$ such that $(\theta \lambda
+1)\gamma \leq 1$ and for any $x,y\in X$
\begin{equation*}
\gamma \lambda d(x,Tx)\leq \left\Vert x-y\right\Vert \ \text{implies that}\
H(bx+Tx,by+Ty)\leq \theta \left\Vert x-y\right\Vert .
\end{equation*}%
Then $Fix(T)$ is nonempty.
\end{thm}

\begin{proof}
We divide the proof into the following two cases.\newline
\textsc{Case 1.} Suppose that $b>0.$ Clearly, $0<\lambda <1.$ In this case,
an implicative contractive condition in the statement of the Theorem reduces
to the following form;
\begin{equation*}
\gamma d(x,T_{\lambda }x)\leq \left\Vert x-y\right\Vert \ \text{implies that}%
\ H(T_{\lambda }x,T_{\lambda }y)\leq \theta \lambda \left\Vert
x-y\right\Vert .
\end{equation*}%
Clearly $T_{\lambda }$ satisfies all the conditions of Theorem $2.1$ of \cite%
{Beg}. Hence $Fix(T)\neq \emptyset .$ \newline
\textsc{Case 2.} Suppose that $b=0.$ In this case, we have $\lambda =1.$
Then $Fix(T)\neq \emptyset $ by Theorem $2.1$ of \cite{Beg}.
\end{proof}

\begin{thm}
\label{TR} Let $(X,||.||)$ be a compact normed space and $T:X\rightarrow
CB(X)$. If there exists $b\in \lbrack 0,\infty )$ and $\gamma \in (0,\frac{1%
}{2}]$ with $\lambda =\frac{1}{b+1}$ such that for any $x,y\in X$
\begin{equation*}
\gamma \lambda d(x,Tx)<\left\Vert x-y\right\Vert \ \ \text{implies\ that}\
H(bx+Tx,by+Ty)<\left\Vert x-y\right\Vert .
\end{equation*}%
Then $Fix(T)$ is nonempty.
\end{thm}

\begin{proof}
We divide the proof into the following two cases.\newline
\textsc{Case 1.} Suppose that $b>0.$ Clearly, $0<\lambda <1.$ In this case,
an implicative contractive condition in the statement of the Theorem reduces
to the following form
\begin{equation*}
\gamma d(x,T_{\lambda }x)<\left\Vert x-y\right\Vert \ \ \text{implies that}\
\ H(T_{\lambda }x,T_{\lambda }y)\leq \lambda \left\Vert x-y\right\Vert
<\left\Vert x-y\right\Vert .
\end{equation*}%
Since $\lambda <1.$ Clearly $T_{\lambda }$ satisfies all the conditions of
Theorem $2.3$ of \cite{Beg}. Hence $Fix(T)\neq \emptyset .$ \newline
\textsc{Case 2.} Suppose that $b=0.$ In this case, we have $\lambda =1.$
Then $Fix(T)\neq \emptyset $ by Theorem $2.3$ of \cite{Beg}.
\end{proof}

\begin{rem}
\ \newline

\begin{enumerate}
\item If we take $T$ a single valued mapping in the Theorem \ref{sds} and
Theorem \ref{TR} then we obtain particular case of Theorem \ref{bot} and
Theorem \ref{main}, respectively.

\item If we put b = 0 in Theorems \ref{sds}, \ref{ghgh} and \ref{TR}, then
we can obtain Theorem $2$ of \cite{KIKKAWA}, Theorem $2.1$ and Theorem $2.3$
of \cite{Beg}, respectively, in the setting of Banach space.
\end{enumerate}
\end{rem}


\begin{thebibliography}{9}
\bibitem{Bamach} S. Banach, Sur les op\'{e}rations dans les ensembles
abstraits et leurs applications aux \'{e}quations int\'{e}grales. Fund.
Math. \textbf{3}, 133-181 (1922)

\bibitem{Beg} I. Beg, S. M. A. Aleomraninejad, Fixed points of Suzuki type
multifunctions on metric spaces, Rend. Circ. Mat. Palermo, \textbf{64}
(2015), 203-207

\bibitem{BP} V. Berinde and M. P$\breve{a}$curar, Approximating fixed points
of enriched contractions in Banach spaces, Journal of Fixed Point Theory and
Applications, \textbf{22} (2), pp. 1-10, 2020.

\bibitem{BP2} V. Berinde, Approximating fixed points of enriched
nonexpansive mappings by Krasnoselskij iteration in Hilbert spaces.
Carpathian J. Math. \textbf{35}(3), 293-304 (2019)

\bibitem{Edel} M. Edelstein, On fixed and periodic points under contractive
mappings, J. London Math. Soc. \textbf{37} (1962) 74-79.

\bibitem{KIKKAWA} M. Kikkawa and T. Suzuki, Three fixed point theorems for
generalized contractions with constants in complete metric spaces, Nonlinear
Analysis: Theory, Methods $\&$ Applications, vol. 69, no. 9, pp. 2942- 2949,
2008

\bibitem{Suzuki} Suzuki, T.: A new type of fixed point theorem in metric
space. Nonlinear Anal. \textbf{71}, 5313-5317 (2009)

\bibitem{SUZUKI 2} T. Suzuki, A generalized Banach contraction principle
that characterizes metric completeness, Proc. Amer. Math. Soc. \textbf{136}
(2008) 1861-1869.
\end{thebibliography}
\end{document}